\documentclass[11 pt, psamsfonts,  leqno]{amsart}
\usepackage{amssymb}
\usepackage{amsmath}	

\def\ignore#1{\relax}

\usepackage[utopia]{mathdesign}

\usepackage{bm}

 \calclayout

\usepackage[plainpages=false, pdfpagelabels,  colorlinks=true, pdfstartview=FitV, linkcolor=blue, citecolor=blue, urlcolor=blue]{hyperref}

\numberwithin{equation}{section}
\numberwithin{figure}{section}

\def\Z{{\mathbb Z}}

\def\eps{\varepsilon}
 
\def\u #1 #2{\mathcal U(#1, #2)}  
\def\uhat #1 #2{\widehat{\mathcal U}(#1, #2)}  

\def\bmw #1{W_{#1}}  

\def\w #1 #2{\bmw {#1}^{(#2)}}  
\def\V #1 #2{V_{#1}^{(#2)}}  
\def\k #1 #2{ KT_{#1}^{(#2)}}
 
\def\ubold{{\bm u}}

\def\hods{\unskip\kern.55em\ignorespaces}

\theoremstyle{plain}
\newtheorem{theorem}{Theorem}[section]

\theoremstyle{plain}

\theoremstyle{plain}
\newtheorem{corollary}[theorem]{Corollary}

\theoremstyle{plain}
\newtheorem{lemma}[theorem]{Lemma}

\theoremstyle{definition}
\newtheorem{definition}[theorem]{Definition}
\theoremstyle{definition}

\theoremstyle{definition}
\newtheorem{remark}[theorem]{Remark}

\theoremstyle{remark}

\title[Degenerate Cyclotomic BMW algebras]{Admissibility conditions for Degenerate Cyclotomic BMW algebras}

\author{Frederick M. Goodman}
\address{ Department of Mathematics\\ University of Iowa\\ Iowa
City, Iowa}
\email{ goodman@math.uiowa.edu}

\subjclass[2000]{20C08, 16G99, 81R50}

\begin{document}
 To appear in {\em Communications in Algebra}
 \bigskip 
 
 \begin{abstract}

 We study admissibility conditions for the parameters of degenerate cyclotomic BMW algebras.  We show that the $u$--admissibility condition of  Ariki, Mathas and Rui is equivalent to a simple module theoretic condition.
 \end{abstract}

 \maketitle

 \section{Introduction}
The {\em cyclotomic Birman--Wenzl--Murakami (BMW)  algebras} are BMW analogues of cyclotomic Hecke algebras  ~\cite{ariki-koike, ariki-book},  while the {\em degenerate cyclotomic BMW algebras} are BMW analogues of degenerate cyclotomic Hecke algebras ~\cite{kleshchev-book}.  

The cyclotomic BMW algebras  were defined by 
H\"aring--Oldenburg in ~\cite{H-O2}  and have recently been studied by three groups of mathematicians:
Goodman and   Hauschild--Mosley ~\cite{GH1, GH2, GH3,  goodman-2008},  Rui, Xu, and Si  ~\cite{rui-2008, rui-2008b},   and Wilcox and Yu  ~\cite{Wilcox-Yu, Wilcox-Yu2, Wilcox-Yu3, Yu-thesis}.

Degenerate affine BMW algebras were introduced by Nazarov ~\cite{nazarov}  under the name {\em affine Wenzl algebras}.    
The   cyclotomic  quotients of these algebras were introduced by Ariki, Mathas, and Rui in \cite{ariki-mathas-rui}  and studied further by Rui and Si in ~\cite{rui-si-degnerate}, under the name {\em cyclotomic Nazarov--Wenzl algebras}.
(We propose to refer to these algebras as degenerate  affine  (resp. degenerate cyclotomic)  BMW algebras instead,   to bring the terminology in line with that used for degenerate affine and  cyclotomic Hecke algebras.)

A peculiar feature of the cyclotomic and degenerate cyclotomic BMW  algebras is that it is necessary to impose ``admissibility" conditions on the parameters entering into the definition of the algebras in order to obtain a satisfactory theory.   For the cyclotomic BMW algebras,  two apparently different conditions were proposed, one by Wilcox and Yu  ~\cite{Wilcox-Yu} and another by Rui and Xu  ~\cite{rui-2008}.  We recently showed ~\cite{goodman-admissibility}  that   the two conditions are equivalent.  Moreover, according to ~\cite{Wilcox-Yu},  admissibility  is equivalent to a simple module theoretic condition: the left ideal
$W_2 e$  generated by the ``contraction" $e$ in the two--strand algebra is free of the maximal possible rank.

It is natural to ask for similar results regarding the parameters of degenerate cyclotomic BMW algebras.
Ariki, Mathas and Rui ~\cite{ariki-mathas-rui}  introduced an admissibility condition (called $u$--admissibility) for these algebras, based on a heuristic involving  the rank of the left ideal $W_2 e$ in the two--strand algebra,  but up until now it has not been shown that their condition is equivalent to $W_2 e$ being free of maximal rank.
In this note, we introduce an analogue for the degenerate cyclotomic BMW algebras  of the admissibility condition of Wilcox and Yu ~\cite{Wilcox-Yu}, we show that this condition is equivalent to  $u$--admissibility, and that both conditions are equivalent to $W_2 e$ being free of maximal rank.

\section{Definitions}

\def\W{W}
\def\Waff{\W^{\text{aff}}}
\def\Set[#1]#2|#3|{\Big\{\ #2\ \Big| \
           \vcenter{\hsize #1mm\centering #3}\Big\}}

\def\map#1#2{\,{:}\,#1\!\longrightarrow\!#2}

Fix a positive integer $n$ and a commutative ring $R$ with
multiplicative identity.  Let  $\Omega=\{\omega_a: \ a\ge0\} $  be a sequence of elements of $R$.

\begin{definition}[\protect{Nazarov~\cite{nazarov}}]
\label{Waff relations}
 The  {\em degenerate affine BMW algebra}
$\Waff_n=\Waff_n(\Omega)$ is 
the unital associative $R$--algebra with generators
$\{s_i,e_i, x_j : \ 1\le i<n \text{ and }1\le j\le n \} $
and relations:
\begin{enumerate}
    \item (Involutions)\ \ 
$s_i^2=1$, for $1\le i<n$.
    \item (Affine braid relations)
\begin{enumerate}
\item $s_is_j=s_js_i$ if $|i-j|>1$,
\item $s_is_{i+1}s_i=s_{i+1}s_is_{i+1}$,  for $1\le i<n-1$,
\item $s_ix_j=x_js_i$ if $j\ne i,i+1$.
\end{enumerate}
    \item (Idempotent relations)\ \ 
$e_i^2=\omega_0e_i$, \ \  for $1\le i<n$.
    \item (Commutation relations)
\begin{enumerate}
\item $s_ie_j=e_js_i$, if $|i-j|>1$,
\item $e_ie_j=e_je_i$, if $|i-j|>1$,
\item $e_ix_j=x_je_i$,  if $j\ne i,i+1$,
\item $x_ix_j=x_jx_i$,  for $1\le i,j\le n$.
\end{enumerate}
    \item (Skein relations)\ \ 
        $s_ix_i-x_{i+1}s_i=e_i-1$,  and\ 
        $x_is_i-s_ix_{i+1}=e_i-1$,\  for 
        $1\le i<n$.
    \item (Unwrapping relations)\ \ 
        $e_1x_1^ae_1=\omega_ae_1$, for $a>0$.
    \item (Tangle relations)
\begin{enumerate}
\item $e_is_i=e_i=s_ie_i$,  for $1\le i\le n-1$,
\item $s_ie_{i+1}e_i=s_{i+1}e_i$, and  $e_ie_{i+1}s_i=e_i s_{i+1}$,  for $1\le i\le n-2$,
\item $e_{i+1}e_is_{i+1}=e_{i+1}s_i$,  and $s_{i+1}e_ie_{i+1}=s_i e_{i+1}$,  for $1\le i\le n-2$.
\end{enumerate}
    \item (Untwisting relations)\ \ 
        $e_{i+1}e_ie_{i+1}=e_{i+1}$,  and 
        $e_ie_{i+1}e_i=e_i$, for  
        $1\le i\le n-2$.
    \item (Anti--symmetry relations)\ \ 
        $e_i(x_i+x_{i+1})=0$,  and         $(x_i+x_{i+1})e_i=0$, for 
        $1\le i<n$.
\end{enumerate}
\end{definition}

\begin{definition}[Ariki, Mathas, Rui ~\cite{ariki-mathas-rui}]  \label{cyclo NZ}
Fix an integer $r\ge1$ and elements $u_1,\dots,u_r $ in $R$,  
The {\em degenerate cyclotomic  BMW  algebra} $\W_{r,n}=\W_{r,n}(u_1, \dots, u_r)$ is the
$R$--algebra $$\Waff_n(\Omega)/\langle(x_1-u_1)\dots(x_1-u_r)\rangle.$$
\end{definition}

Note that, due to the symmetry of the relations,  $\Waff_n$  has a unique $R$--linear  algebra involution $*$
(that is, an algebra anti-automorphism of order $2$)    
 such that $e_i^* = e_i$,  $s_i^* = s_i$,  and $x_i^* = x_i$   for all $i$.  The involution passes to cyclotomic quotients.

\begin{lemma}[see ~\cite{ariki-mathas-rui}, Lemma 2.3]                        \label{SX^a}
In the degenerate affine BMW algebra $\Waff_n$,   for  $1\le i<n$ and  $a\ge1$, one has
\begin{equation} \label{equation: S X^a}
s_ix_i^a=x_{i+1}^as_i+\sum_{b=1}^ax_{i+1}^{b-1}(e_i-1)x_i^{a-b}.
\end{equation}
\end{lemma}

Taking $i = 1$ in Lemma \ref{SX^a},  pre-- and post--multiplying by $e_1$ and simplifying using the relations gives:
\begin{equation} \label{equation: weak admissibility1}
\omega_a e_1  =(-1)^a\omega_a e_1
         +\sum_{b=1}^a (-1)^{b-1}\omega_{b-1}\omega_{a-b}e_1
         +\sum_{b=1}^a(-1)^b\omega_{a-1}e_1
\end{equation}

For $a$ odd, this gives
\begin{equation} \label{equation: weak admissibility2}
2 \omega_a e_1  =
         \sum_{b=1}^a (-1)^{b-1}\omega_{b-1}\omega_{a-b}e_1
        - \omega_{a-1}e_1,
\end{equation}
which is Corollary 2.4 in ~\cite{ariki-mathas-rui}.    As noted in ~\cite{ariki-mathas-rui},   the identity derived from
(\ref{equation: weak admissibility1})  in case $a$ is even is a tautology.  

Consider the cyclotomic algebra $W_{r, n}(u_1, \dots, u_r)$,  and let $a_j$ denote the signed elementary symmetric function in $u_1, \dots, u_r$,   namely, $a_j = (-1)^{r -j}  \eps_{r-j}(u_1, \dots, u_r)$.     Thus, in the cyclotomic algebra, we have the relation $\sum_{j = 0}^r  a_j  x_1^j = 0$.    Multiplying by $x_1^a$  for an arbitrary $a \ge 0$  and pre-- and post--multiplying by $e_1$  gives
\begin{equation}   \label{equation: weak admissibility3}
\sum_{j = 0}^r  a_j  \omega_{j + a} e_1 = 0.
\end{equation}

\begin{corollary}  \label{corollary:   weak admissibility conditions}
Consider the cyclotomic algebra $W_{r, n}(u_1,  \dots, u_r)$.    If $e_1$ is not a torsion element over $R$,  then we have:
\begin{enumerate}
\item  $2 \omega_a   =
         \sum_{b=1}^a (-1)^{b-1}\omega_{b-1}\omega_{a-b}
        - \omega_{a-1}$,  for  all  odd $a \ge 1$,  and
\item $ \sum_{j = 0}^r  a_j  \omega_{j + a}  = 0$,  for all $a \ge 0$.  
\end{enumerate}
\end{corollary}

\begin{definition}  Say that the parameters $\omega_a$   ($a \ge 0$)  and $u_1,  \dots, u_r$ are {\em weakly admissible},   or that the ground ring $R$ is {\em weakly admissible},  if the relations of Corollary \ref{corollary:   weak admissibility conditions} hold.
\end{definition}

Weak admissibility is a non--triviality condition for the cyclotomic algebras;  if the ground ring is a field, and weak admissibility fails, then $e_1 = 0$,  and the cyclotomic algebra reduces to a specialization of the degenerate cyclotomic Hecke algebra, see ~\cite{ariki-mathas-rui}, pages 60--61.

In the following, we use the notation $\delta_{(P)} = 1$ if $P$ is true and  $\delta_{(P)} = 0$ if $P$ is false.

\begin{lemma}  In the degenerate affine BMW algebra, for $a \ge 1$,  we have
\begin{equation} \label{equation: s x^a e}
s_1 x_1^a e_1 = (-1)^a x_1^a e_1 + \sum_{b = 1}^a (-1)^{b-1} \omega_{a-b} x_1^{b-1} e_1   -  \delta_{(\text{$a$ is odd})}\ x_1^{a-1} e_1.
\end{equation}
\end{lemma}

\begin{proof}  Take $i = 1$ in equation  (\ref{equation: S X^a}).  Post--multiply   by $e_1$, and simplify,  using 
the relations.  
\end{proof} 
 
 \section{$u$--admissibility}
 
 The definition of $u$--admissibility is motivated by Theorem \ref{theorem:  ariki mathas rui} below,  which is essentially contained in  ~\cite{ariki-mathas-rui}, although not explicitly stated there.  
 
 \begin{lemma} \label{lemma: spanning lemma}
  Let $R$ be any ground ring with parameters $\omega_a$ for $a \ge 0$ and
 $u_1, \dots, u_r$.   Let $\bmw {2, R}$  denote the two strand degenerate cyclotomic BMW algebra defined over $R$.  Then 
 \begin{enumerate}
 \item The left
  ideal $\bmw {2, R} \ e_1$  equals the $R$--span of $\{ e_1, x_1 e_1, \dots, x_1^{r-1} e_1\}$. 
  \item $\bmw {2, R}$ is spanned over $R$ by   $\{ x_1^a e_1 x_1^b,\  x_1^a x_2^b s_1, \ x_1^a x_2^b  : 0 \le a, b \le r-1\} $
  \end{enumerate}
\end{lemma}  

\begin{proof}  Using Lemma \ref{SX^a},  and the defining relations of the algebra,  one sees that the
span of $\{ e_1, x_1 e_1, \dots, x_1^{r-1} e_1\}$ is invariant under left multiplication by the generators $x_1, e_1, s_1$,   and
that  $x_2  x_1^a e_1 = -x_1^{a + 1} e_1$.    This proves part (1).  Part (2) is similar, see ~\cite{ariki-mathas-rui},  Proposition 2.15.
\end{proof}

 \begin{theorem}[\cite{ariki-mathas-rui}]  \label{theorem:  ariki mathas rui}
  Let $F$ be  a field of characteristic $\ne 2$, with parameters
 $\omega_a$ for $a \ge 0$  and $u_1, \dots, u_r$.  Assume that the $u_i$ are distinct and
 $u_i + u_j \ne 0$ for all $i, j$.    Let $W = W_{2, F}$  be the degenerate cyclotomic BMW algebra defined over $F$ with parameters $\omega_a$ for $a \ge 0$  and $u_1, \dots, u_r$.
Then the following conditions are equivalent:
 \begin{enumerate}
 \item $\{e_1, x_1 e_1, \dots, x_1^{r-1} e_1\} \subseteq  W e_1$ is linearly independent over $F$, and $e_1 W e_1 \ne 0$.    
 \item For all $a \ge 0$,   $\omega_a = \sum_{i = 1}^r  \gamma_i u_i^a$,  where
\begin{equation} \label{equation: definition of the gammas}
\gamma_i=(2u_i-(-1)^r)\displaystyle
        \prod_{j\ne i}\dfrac{u_i+u_j}{u_i-u_j},
  \end{equation}
        and some $\omega_a$ is non--zero. 
  \item   $W$  admits a module $M$  with basis   $\{ v_0, x_1 v_0, \dots, x_1^{r-1} v_0\}$
 such that $e_1 M  = F v_0$.
  \ignore{
 \item  $F$ is weakly admissible, and $W$  admits a module $M$  with basis \break  $\{ v_0, x_1 v_0, \dots, x_1^{r-1} v_0\}$ such $e_1$,  $s_1$, and $x_2$  act by the formulas: $e_1 x_1^j v_0 = \omega_j v_0$,  $x_2 x_1^j v_0 = - x_1^{j+1} v_0$, and 
 \begin{equation}
s_1 x_1^a v_0 = (-1)^a x_1^a v_0 + \sum_{b = 1}^a (-1)^{b-1} \omega_{a-b} x_1^{b-1} v_0   -  \delta_{(\text{$a$ is odd})}\ x_1^{a-1} e_1. 
 \end{equation}
}

\end{enumerate}
 \end{theorem}
 
 \begin{proof} The implication (1) $\implies$ (3) is obvious.  
 
 Assume condition (3).  We have $v_0 = e_1 m$  for some $m\in M$,  so
 $e_1 x^j v_0 = e_1 x^j e_1 m = \omega_j e_1 m = \omega_j v_0$  for $1 \le j \le r-1$. 
 Moreover, some $\omega_j \ne 0$  since $e_1 M \ne (0)$.  
 Define $p_i \in \bmw  {1, F}$ by $p_i = \displaystyle \prod_{j \ne i} \frac{x_1 - u_j}{u_i - u_j}$.  Then $p_i^2 = p_i$,  $\sum_i p_i = 1$,  and $x_1 p_i = u_i p_i$.   
 Define $m_i \in M$ by $m_i = p_i v_0$.   Then $m_i \ne 0$ by the assumed linear independence of  $\{ v_0, x_1 v_0, \dots, x_1^{r-1} v_0\}$,   $x_1 m_i = u_i m_i$, and
 $\sum_i m_i = (\sum_i p_i) v_0 = v_0$.     It follows that $\{m_1, \dots, m_r\}$ is linearly independent, since the $m_i$ are eigenvectors for $x_1$  with distinct eigenvalues.   Define $\kappa_j$ and $c_{i, j}$ in $F$ by
 $e_1 m_j = \kappa_j v_0 = \kappa_j \sum_i m_i$,  and $s_1 m_j = \sum_i c_{i, j} m_i$.    (It will be shown that 
 $\kappa_j = \gamma_j$,  where $\gamma_j$ is defined above.)    Note that $e_1 M \ne (0)$ implies that $\kappa_j \ne 0$ for some $j$. 
 
 The argument continues as in the proof of Theorem 3.2 in ~\cite{ariki-mathas-rui},  pp. 65--67.   Apply the identity
 $x_1 s_1 - s_1 x_2 - e_1 + 1 = 0$ to $m_j$ to derive a formula for $c_{i, j}$, 
 $$
 c_{i, j} =  (\kappa_j - \delta_{i, j})/(u_i + u_j).
 $$
 Next apply the identity $e_1 = s_1 e_1$ to $m_i$ to get 
 \begin{equation} \label{equation: derive formula for  gammas 1}
 \kappa_i\sum_{j=1}^d m_j=e_1 m_i=s_1 e_1 m_i
       =\kappa_i\sum_{j=1}^d\Big\{\frac{\kappa_j-1}{2u_j}
            +\sum_{k\ne j}\frac{\kappa_k}{u_j+u_k}\Big\}m_j,
 \end{equation}
for $i=1,\dots,r$.  Since at least one $\kappa_i$ is non--zero,  matching coefficients in  
(\ref{equation: derive formula for  gammas 1}) gives the equations
\begin{equation}  \label{equation: derive formula for  gammas 2}
\sum_{k=1}^d\frac{\kappa_k}{u_j+u_k}=1+\frac1{2u_j},
\end{equation}
for $j=1,\dots,r$.  Now it is shown in ~\cite{ariki-mathas-rui}, page 66, that the unique solution to this system of equations is $\kappa_j = \gamma_j$  for $1 \le j \le 1$.   Finally, we have
\begin{equation}
\omega_j v_0 = e_1 x_1^j v_0 = e_1 x_1^j (\sum_i m_i) = e_1 \sum_i u_i^j m_i =
(\sum_i  u_i^j  \gamma_i)  v_0.
\end{equation}
This shows (3) $\implies$ (2).  

Finally,  (2) $\implies$ (1)  by  Theorem A of ~\cite{ariki-mathas-rui}, namely  assuming
(2),  
$W$ has an $R$--basis that includes $\{e_1, x_1 e_1, \dots, x_1^{r-1} e_1\}$, so the latter set is linearly independent.  
 \end{proof}  
 
 The elements $\gamma_j$ appearing in Theorem \ref{theorem:  ariki mathas rui} are rational functions in $u_1, \dots, u_r$ with singularities at $u_i = u_j$,  but it is shown in ~\cite{ariki-mathas-rui} that the elements
$\sum_i \gamma_i u_i^a$  are polynomials in $u_1, \dots, u_r$, as follows:  

Let $\ubold_1, \dots, \ubold_r$ and $t$  be algebraically independent indeterminants over $\Z$.   Define symmetric polynomials $q_a(\ubold)$  in $\ubold_1, \dots, \ubold_r$ by
$$\prod_{i=1}^r\frac{1+\ubold_it}{1-\ubold_it}=\sum_{a\ge0}q_a(\ubold)t^a.$$
The polynomials $q_a$  are known as {\em Schur  $q$--functions}.
Let $\gamma_j(\ubold)$  be defined by  (\ref{equation: definition of the gammas})  with
$u_i$  replaced by $\ubold_i$.   Moreover,  let $\eta_a(\ubold) = \sum_{i = 1}^r  \gamma_j(\ubold) \ubold_j^a$  for $a \ge 0$.   Then  (\cite{ariki-mathas-rui},  Lemma 3.5)
\begin{equation}
\eta_a(\ubold) = q_{a+1}(\ubold) - \frac{1}{2}(-1)^r q_a(\ubold) + \frac{1}{2}\delta_{a, 0}. 
\end{equation}
In particular the $\eta_a$ are polynomials in $\ubold_1, \dots, \ubold_r$.  

\begin{corollary}  Let $F$ be  a field of characteristic $\ne 2$, with parameters
 $\omega_a$ for $a \ge 0$  and $u_1, \dots, u_r$.  Assume that the $u_i$ are distinct and
 $u_i + u_j \ne 0$ for all $i, j$.    Let $W = W_{2, F}$  be the degenerate cyclotomic BMW algebra defined over $F$ with parameters $\omega_a$ for $a \ge 0$  and $u_1, \dots, u_r$.
If $\{e_1, x_1 e_1, \dots, x_1^{r-1} e_1\} \subseteq  W e_1$ is linearly independent over $F$, and $e_1 W e_1 \ne 0$, then    
\begin{equation}
\omega_a = q_{a+1}(u_1, \dots, u_r) - \frac{1}{2}(-1)^r q_a(u_1, \dots, u_r) + \frac{1}{2}\delta_{a, 0}. 
\end{equation}
\end{corollary}

This motivates the following definition, which makes sense for arbitrary $u_1, \dots, u_r$: 

\begin{definition}[\cite{ariki-mathas-rui}]  Let $R$ be a commutative ring with parameters $\omega_a$ ($a \ge 0$) and $u_1, \dots, u_r$.  Suppose that $2$ is invertible in $R$.   Say that the parameters are $u$--admissible if 
\begin{equation} \label{equation: u admissibility relations}
\omega_a = q_{a+1}(u_1, \dots, u_r) + \frac{1}{2} (-1)^{r-1}q_a(u_1, \dots, u_r)  + \frac{1}{2} \delta_{a, 0}
\end{equation}
 for all $a \ge 0$. 
\end{definition}

\section{Admissibility}    We fix a ground ring $R$ with parameters $\omega_a$   ($a \ge 0$)  and
$u_1, \dots, u_r$.     We consider the two strand degenerate cyclotomic BMW algebra over $R$,
$W = W_{2, r}(u_1, \dots, u_r)$ and we write $e = e_1$,  $s = s_1$,  and $x = x_1$.

\begin{lemma}  \label{lemma: linear independence implies admissibility}
Suppose that $\{ e, x e, \dots, x^{r-1} e\}$ is linearly independent over $R$.  Then the parameters $\omega_a$  ($a \ge 0$)  and $u_1, \dots, u_r$  are weakly admissible and satisfy the following relations:
\begin{equation}  \label{equation:  admissibility relation}
 \sum_{\mu = 0}^{r-j-1}  \omega_\mu  a_{\mu + j +1}  = - 2 \delta_{(\text{$r-j$ is odd})}\  a_j  + \delta_{(\text{$j$ is even})} \  a_{j+1},
\end{equation}
for $0 \le j \le r-1$. 
\end{lemma}

\begin{proof}  Since $\{ e, x e, \dots, x^{r-1} e\}$ is assumed linearly independent over $R$,  in particular
$e$ is not a torsion element over $R$,  and hence $R$ is weakly admissible by Corollary \ref{corollary:   weak admissibility conditions}.  

If $r  = 1$,   (\ref{equation:  admissibility relation})  reduces to the single equation
$\omega_0 + 2 a_0 - 1 = 0$,  which follows from  $(s x + xs + 1-e) e = 0$,  together with
$x = u_1 = -a_0$  and $s e = e$.  
Assume $r \ge 2$.  We have
$$
0 = (s x - x_2 s + 1 -e ) x^{r-1} e =  (s x + x s + 1 - e) x^{r-1} e.
$$
Apply the identity $x (x^{r-1} e) =  -\sum_{j = 0}^{r-1} a_j x^j e$ as well as the identity (\ref{equation: s x^a e}) and simplify.    This gives:
$$
\begin{aligned}
0= &-a_0 e - \sum_{j=1}^{r-1} (-1)^j a_j x^j e  + \sum_{\substack{0 \le j \le r-2\\ j \text{ even }}} a_{j+1} x^j e
+ \sum_{j=0}^{r-2} (-1)^j \left(\sum_{\mu = 0}^{r-j-2} \omega_\mu a_{\mu + j + 1} \right ) x^j e \\
&+ (-1)^r \sum_{j= 0}^{r-1} a_j x^j e - \delta_{(\text{$r$ is even})}\ x^{r-1} e + \sum_{j =1}^{r-1} \omega_{r-1-j} x^je \\
&+ x^{r-1} e -  \omega_{r-1} e,
\end{aligned}
$$
where the three lines of the display correspond to evaluation of $s x x^{r-1} e$,   $x s x^{r-1} e$,  and
$(1 -e) x^{r-1} e$.    Because $\{ e, x e, \dots, x^{r-1} e\}$ is assumed to be linearly independent,  the coefficient of $x^j e$ is zero for each $j$,   $0 \le j \le r-1$.     Extracting the coefficients yields (\ref{equation:  admissibility relation}).    Here one has to treat the three cases $j = 0$,  \ $1 \le j \le r-2$,  and $j = r-1$  separately,  but  the result in all three cases  is the same.  
\end{proof}

\begin{definition} Say that the parameters $\omega_a$  ($a \ge 0$) and $u_1, \dots, u_r$  are {\em admissible}
(or that the ground ring $R$ is admissible)  if  the relations (\ref{equation:  admissibility relation}) hold for
$0 \le j \le r-1$  and   $ \sum_{\mu = 0}^r  a_\mu  \omega_{\mu + a}  = 0$ holds   for all $a \ge 0$. 
\end{definition}

\begin{remark}  Admissibility is analogous to the admissibility condition of Wilcox and Yu for the cyclotomic BMW algebras ~\cite{Wilcox-Yu}.   Our terminology differs from that in ~\cite{ariki-mathas-rui},  where admissibility means essentially what we have called weak admissibility.  
\end{remark}

\begin{lemma} \label{lemma: universal polynomials}
 There exist universal polynomials $H_a(\ubold_1, \dots, \ubold_r)$  ($a \ge 0$),  symmetric in $\ubold_1, \dots \ubold_r$,    such that whenever
$R$ is an admissible ring, one has 
\begin{equation} \label{equation:  omegas in terms of a's 1}
\omega_a = \break H_a(u_1, \dots, u_r)
\end{equation}
  for $a \ge 0$.    
\end{lemma}

\begin{proof}   The system of relations (\ref{equation:  admissibility relation}) is a unitriangular linear system of equation for the variables $\omega_0,  \dots, \omega_{r-1}$.  In fact, if we list the equations in reverse order
then the matrix of coefficients is
$$
\begin{bmatrix}
1 & & & &\\
a_{r-1} & 1 & && \\
a_{r-2} & a_{r-1} & 1 & &\\
\vdots  & & \ddots& \ddots & \\
a_1   &a_2  &\dots&a_{r-1}& 1\\
\end{bmatrix}.
$$
Solving the system for $\omega_0,  \dots, \omega_{r-1}$  gives  these quantities as polynomial functions of 
$a_0,  \dots, a_{r-1}$,  thus symmetric polynomials in $u_1, \dots, u_r$.    The relations
 $ \sum_{j = 0}^r  a_j  \omega_{j + m}  = 0$,  for all $m \ge 0$ yield (\ref{equation:  omegas in terms of a's 1})  for
 $a \ge r$.  
\end{proof}

\section{Equivalence of admissibility conditions}

In this section we will show that admissibility and $u$--admissibility are equivalent (assuming that $2$ is invertible in the ground ring, so that $u$--admissibility makes sense.)  

First, we will obtain the polynomials $H_a$ of Lemma \ref{lemma: universal polynomials} explicitly in terms of the Schur $q$--functions.
Considering the generating function for the Schur $q$--functions, 
\begin{equation} \label{equation:  Schur q functions 1}
\prod_{i = 1}^r \frac{1 + \ubold_i t}{1 - \ubold_i t} =   \sum_{a \ge 0} q_a(\ubold) t^a,
\end{equation}
 we have
\begin{equation} \label{equation:  Schur q functions 2}
\left(\prod_{i = 1}^r (1 - \ubold_i t)\right )   \left(\sum_{a \ge 0} q_a(\ubold) t^a\right ) = \prod_{i = 1}^r (1 + \ubold_i t).
\end{equation}
Taking into account that $q_0(\ubold) = 1$, we also have
\begin{equation} \label{equation:  Schur q functions 3}
\left(\prod_{i = 1}^r (1 - \ubold_i t)\right )   \left(\sum_{a \ge 1} q_a(\ubold) t^a\right ) = \prod_{i = 1}^r (1 + \ubold_i t) -   \prod_{i = 1}^r (1 - \ubold_i t).
\end{equation}
Matching coefficients in (\ref{equation:  Schur q functions 2}) and writing in terms of the signed elementary symmetric functions $a_i(\ubold)$ gives
\begin{equation}  \label{equation:  Schur q functions 4}
\sum_{\mu = 0}^{r-j-1} q_\mu(\ubold)  a_{\mu + j + 1}(\ubold) =   (-1)^{r - j -1}  a_{j+1}(\ubold), \quad \text{for\  $0 \le j \le r-1$,}
\end{equation}
and, moreover, 
\begin{equation} \label{equation: Schur q functions 8}
 \sum_{\mu = 0}^r  a_\mu(\ubold)  q_{\mu + a}(\ubold)  = 
\begin{cases}   (-1)^r a_0(\ubold) & \text{if $a = 0$,}
\\
0 &\text{if $a \ge 1$.}
\end{cases}
\end{equation}
Doing the same with    (\ref{equation:  Schur q functions 3}) yields
\begin{equation}  \label{equation:  Schur q functions 5}
\sum_{\mu = 0}^{r-j-1} q_{\mu+1}(\ubold)  a_{\mu + j + 1}(\ubold) =  -2 \delta_{(\text{$r-j$ is odd})} \  a_{j}(\ubold), \quad \text{for $0 \le j \le r-1$.}
\end{equation}
If we set $$\eta_{a}(\ubold) =  q_{a+1}(\ubold) + \frac{1}{2} (-1)^{r-1}q_a(\ubold)  + \frac{1}{2} \delta_{a, 0},$$ then, using (\ref{equation:  Schur q functions 4}) and (\ref{equation:  Schur q functions 5}), we get
\begin{equation}  \label{equation:  Schur q functions 6}
\sum_{\mu = 0}^{r-j-1} \eta_{\mu}(\ubold)  a_{\mu + j + 1}(\ubold) =  -2 \delta_{(\text{$r-j$ is odd})} \  a_{j}(\ubold)
 + \delta_{(\text{$j$ is even})} \  a_{j+1}(\ubold),
\end{equation}
for $0 \le j \le r-1$.  
From (\ref{equation: Schur q functions 8}), we  obtain
\begin{equation} \label{equation:  Schur q functions 7}
\sum_{\mu = 0}^r  a_\mu(\ubold)  \eta_{\mu + a}(\ubold) = 0, \quad \text{for all $a \ge 0$}.
\end{equation}

\begin{lemma}  \label{lemma:  admissible if and only if u admissible}
Let $R$ be commutative ring in which $2$ is invertible.  Let
$\omega_a$  ($a \ge 0$)  and $u_1, \dots, u_r$  be parameters in $R$.     Then  the parameters are
admissible, if, and only if, they are $u$--admissible.  
 
\end{lemma}

\begin{proof}  By definition, the parameters are $u$--admissible if $\omega_a = \eta_a(u_1, \dots, u_r)$  for all 
$a \ge 0$.  It follows from
(\ref{equation:  Schur q functions 6})  and (\ref{equation:  Schur q functions 7})  that $u$--admissible parameters are admissible.

On the other hand, if the parameters are admissible, then
the relations (\ref {equation:  admissibility relation}) for $0 \le j \le r-1$   and  $ \sum_{\mu = 0}^r  a_\mu  \omega_{\mu + a}  = 0$  for $a \ge 0$  uniquely determine the $\omega_a$ for all $a \ge 0$ as symmetric polynomial functions
of $u_1, \dots, u_r$.  But according to (\ref{equation:  Schur q functions 6}) and (\ref{equation:  Schur q functions 7}),   the elements $\eta_a(u_1, \dots, u_r)$  satisfy the same relations.  Hence  $\omega_a = \eta_a(u_1, \dots, u_r)$  for $a \ge 0$,  so the parameters are $u$--admissible.  
\end{proof}

\vbox{
\begin{theorem} Let $R$ be a commutative ring with parameters $\omega_a$  ($a \ge 0$)  and 
$u_1, \dots, u_r$.   Suppose that \ $2$ is invertible in $R$.  
Consider the two strand degenerate cyclotomic BMW algebra over $R$,
$W = W_{2, r}(u_1, \dots, u_r)$.
The following are equivalent:
\begin{enumerate}
\item  $\{ e_1, x_1 e_1, \dots, x_1^{r-1} e_1\}$ is linearly independent over $R$.
\item $\{ x_1^a e_1 x_1^b, x_1^a x_2^b s_1, x_1^a x_2^b  : 0 \le a, b \le r-1\} $ is linearly independent over $R$.  
\item  The parameters are admissible.
\item The parameters are $u$--admissible.
\end{enumerate}
\end{theorem}
}

\begin{proof}  
 Lemma  \ref{lemma: linear independence implies admissibility} gives (1) $\implies$ (3). 
Lemma \ref{lemma:  admissible if and only if u admissible} gives (3)  $\Longleftrightarrow$  (4).  
The implication  (4) $\implies$  (2)  is part of  the main result (Theorem A)  of  ~\cite{ariki-mathas-rui}.  Finally (2) $\implies$ (1) is trivial.  
\end{proof}

If the equivalent conditions of the theorem hold,  then the sets in (1) and (2) are $R$--bases
of $\bmw {2, R} e_1$,  respectively of $\bmw {2, R}$,  since they are spanning by Lemma \ref{lemma: spanning lemma}.   If $R$ is an integral domain the conditions are equivalent to:
(1$'$)  $\bmw {2, R}e_1$ is free over $R$ of rank $r$,  respectively 
(2$'$)  $\bmw {2, R}$ is free over $R$ of rank $3 r^2$.

\bibliographystyle{amsplain}
\bibliography{admissibility} 

\end{document}